\documentclass[12pt]{article}
\usepackage[utf8]{inputenc}
\usepackage{eurosym}
\usepackage{epic,latexsym,amssymb,amsmath}
\usepackage{tikz}
\usepackage{amsmath,amsfonts,latexsym, amssymb, amsthm}
\usepackage[mathscr]{euscript}
\usepackage{mathrsfs}
\usepackage{graphics}
\usepackage{float}
\usepackage{stmaryrd}
\usepackage{color}
\usepackage{geometry}
\usepackage{lscape}
\usepackage{comment}
\usepackage[pdfa]{hyperref}
\newtheorem{thm}{Theorem}[section]
\newtheorem{prop}[thm]{Proposition}
\newtheorem{lemma}[thm]{Lemma}
\newtheorem*{lemma*}{Lemma}
\newtheorem{cor}[thm]{Corollary}

\newtheorem{claim}[thm]{Claim}

\newtheorem{definition}[thm]{Definition}

\newcommand{\im}{\text{im}}
\newcommand{\ext}{\text{ex}}
\newcommand{\Hom}{\text{Hom}}

\DeclareMathOperator{\OPT}{OPT}
\DeclareMathOperator{\Dist}{Dist}

\renewenvironment{proof}[1][Proof]{\textbf{#1.} }{\ \rule{0.5em}{0.5em}}

\newcommand{\vc}[1]{\ensuremath{\vcenter{\hbox{#1}}}}

\newcommand{\ftwa}{\vc{\includegraphics[page=1, scale = 0.5]{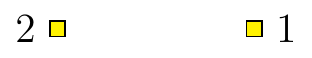}}}
\newcommand{\ftwb}{\vc{\includegraphics[page=2, scale = 0.5]{figs2.pdf}}}

\newcommand{\ftra}{\vc{\includegraphics[page=1, scale = 0.5]{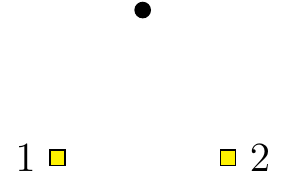}}}
\newcommand{\ftrb}{\vc{\includegraphics[page=2, scale = 0.5]{figs3.pdf}}}
\newcommand{\ftrc}{\vc{\includegraphics[page=3, scale = 0.5]{figs3.pdf}}}
\newcommand{\ftrd}{\vc{\includegraphics[page=4, scale = 0.5]{figs3.pdf}}}
\newcommand{\ftre}{\vc{\includegraphics[page=5, scale = 0.5]{figs3.pdf}}}
\newcommand{\ftrf}{\vc{\includegraphics[page=8, scale = 0.5]{figs3.pdf}}}

\newcommand{\ffra}{\vc{\includegraphics[page=1, scale = 0.5]{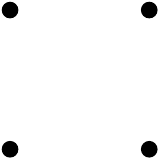}}}
\newcommand{\ffrb}{\vc{\includegraphics[page=2, scale = 0.5]{figs4.pdf}}}
\newcommand{\ffrc}{\vc{\includegraphics[page=3, scale = 0.5]{figs4.pdf}}}
\newcommand{\ffrd}{\vc{\includegraphics[page=4, scale = 0.5]{figs4.pdf}}}
\newcommand{\ffre}{\vc{\includegraphics[page=5, scale = 0.5]{figs4.pdf}}}
\newcommand{\ffrf}{\vc{\includegraphics[page=6, scale = 0.5]{figs4.pdf}}}
\newcommand{\ffrg}{\vc{\includegraphics[page=7, scale = 0.5]{figs4.pdf}}}
\newcommand{\ffrh}{\vc{\includegraphics[page=8, scale = 0.5]{figs4.pdf}}}
\newcommand{\ffri}{\vc{\includegraphics[page=9, scale = 0.5]{figs4.pdf}}}
\newcommand{\ffrj}{\vc{\includegraphics[page=10, scale = 0.5]{figs4.pdf}}}
\newcommand{\ffrk}{\vc{\includegraphics[page=11, scale = 0.5]{figs4.pdf}}}
\newcommand{\figsfourlabelone}{\vc{\includegraphics[page=12, scale = 0.5]{figs4.pdf}}}
\newcommand{\figsfourlabeltwo}{\vc{\includegraphics[page=13, scale = 0.5]{figs4.pdf}}}

\newcommand{\ffiveone}{\vc{\includegraphics[page=1, scale = 0.5]{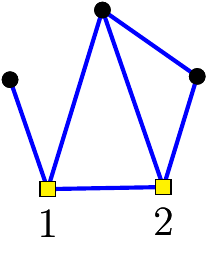}}}

\tikzset{vtx/.style={inner sep=1.7pt, outer sep=0pt, circle, fill,draw}}
\tikzset{gedge/.style={solid,color=black,line width=1.2pt,opacity=0.75}}

\begin{document}

\title{Paths of Length Three are $K_{r+1}$-Tur\'an-Good}
\date{\today}
\author{Kyle Murphy\\
\small{Iowa State University}\\
\small{\texttt{kylm2@iastate.edu}}
 \and 
JD Nir\\
\small{University of Manitoba}\\
\small{\texttt{jd.nir@umanitoba.ca}}
}
\maketitle

\begin{abstract}
The generalized Tur\'an problem $\ext(n,T,F)$ is to determine the maximal number
of copies of a graph $T$ that can exist in an $F$-free graph on $n$ vertices.
Recently, Gerbner and Palmer noted that the solution to the generalized Tur\'an
problem is often the original Tur\'an graph. They gave the name
``$F$-Tur\'an-good'' to graphs $T$ for which, for large enough $n$, the solution
to the generalized Tur\'an problem is realized by a Tur\'an graph. They prove
that the path graph on two edges, $P_2$, is $K_{r+1}$-Tur\'an-good for all $r
\ge 3$, but they conjecture that the same result should hold for all $P_\ell$.
In this paper, using arguments based in flag algebras, we prove that the path on
three edges, $P_3$, is also $K_{r+1}$-Tur\'an-good for all $r \ge 3$.
\end{abstract}

\section{Introduction}

One of extremal graph theory's most celebrated results was introduced in
\cite{Turan} by Tur\'an who asked how many edges a (simple) graph on $n$ vertices can contain if
it has no clique containing $r+1$ vertices. Tur\'an's solution, which we
denote $\ext(n, K_{r+1})$, is asymptotically $(1-\frac{1}{r})\binom{n}{2}$. Additionally, Tur\'{a}n showed that the unique extremal graph is the complete $r$-partite graph on $n$
vertices with parts of size $\lceil \frac{n}{r} \rceil$ or $\lfloor \frac{n}{r}
\rfloor$ (so that no pair of parts differs in size by more than one). We call this
graph the \emph{Tur\'an graph} and denote it $T_r(n)$.

The first extensions to Tur\'an's theorem considered forbidding graphs other
than cliques. For any graph $F$, we say a graph $G$ is \emph{$F$-free} if it
contains no (not necessarily induced) subgraph isomorphic to $F$. We use
$\ext(n,F)$ to denote the maximal number of edges in an $F$-free graph on $n$
vertices. The general case is solved asymptotically by the
Erd\H{o}s-Stone-Simonovits Theorem \cite{ErdosSimonovits} which proves
\[ \ext(n,F) = \left(1-\frac{1}{\chi(F)-1}+o(1)\right)\binom{n}{2}. \]

To further generalize the problem, one may consider counting subgraphs other
than edges. Let $\nu(T,G)$ denote the number of distinct, not necessarily
induced subgraphs of $G$ isomorphic to $T$. We denote by $\ext(n,T,F)$ the
maximum of $\nu(T,G)$ over all $F$-free graphs $G$ on $n$ vertices. (Here $T$ is
the ``target'' graph while $F$ is ``forbidden.'') The first question of this
form to be resolved was due to Zykov in 1949~\cite{Zykov} who determined the
value of the function $\ext(n,K_t,K_r)$ when $t < r$ by proving that the
Tur\'{a}n graph is the unique extremal graph. 

\begin{thm}[Zykov~\cite{Zykov}]\label{thm_zykov_bound}
Let $r$ and $t$ be integers such that $t < r$. Then for all $n$, the Tur\'{a}n
graph $T_t(n)$ is the unique $K_r$-free graph on $n$ vertices containing the
maximum number of $K_t$ subgraphs. 
\end{thm}

Several sporadic cases were investigated (see, for example,
\cite{BollobasGyori, Gyori}) before 2015 when Alon and Shikhelman introduced a
systematic study in \cite{Alon} in which they determine, among other
results, that for forbidden graphs $F$ with $\chi(F) = k+1 > r$,
\[ \ext(n,K_r,F) = (1+o(1))\binom{k}{r}\left(\frac{n}{k}\right)^r. \]
A more precise result can be found in \cite{MaQiu}. Since then, the area has
been widely studied; see \cite{CutlerNirRadcliffe, GERBNER2019103001, HalfpapPalmer,
LohTaitTimmonsZhou, Luo} for an (incomplete) sampling of authors and results.

As in the original Zykov result, for many choices of $T$ and $F$ the Tur\'an graph
emerges as the optimal graph, at least for large enough $n$. In
\cite{GerbnerPalmer2020}, Gerbner and Palmer introduced the term
$F$-Tur\'an-good to describe such target graphs $T$:

\begin{definition}
Fix an $(r+1)$-chromatic graph $F$ and a graph $T$ that does not contain $F$ as a
subgraph. We say that $T$ is \emph{$F$-Tur\'an-good} if $\ext(n,T,F) =
\nu(T,T_r(n))$ for every $n$ large enough.
\end{definition}

In the same paper, Gerbner and Palmer prove that the path graph on $\ell$ edges,
$P_\ell$, is $K_{r+1}$-Tur\'an-good for $\ell = 2$ and $r \ge 3$. They
conjecture that paths should be Tur\'an-good for all choices of $r$ and $\ell$.
In this paper we establish that $P_3$, the path on three edges, is
$K_{r+1}$-Tur\'an-good for all $r \ge 3$.

To be precise,  define the \emph{density} of $H$ in $G$ to be
\[ d(H,G) = \nu(H,G) \binom{|G|}{|H|}^{-1} \]
and let $\mathcal{F}_{n,r}$ be the family of $K_{r+1}$-free graphs on $n$
vertices. We define
\[ \OPT_r(P_3) = \lim\limits_{n \to \infty} \max_{G_n \in \mathcal{F}_{n,r}} d(P_3,G_n).\]
Then the following theorem is the primary result of this paper:

\begin{thm}\label{main_p4_thm}
For any integer $r \geq 3$,
\begin{enumerate}
    \item[ (i) ] $\OPT_r(P_3) = 12\left(\frac{r-1}{r}\right)^3$.
    \item[ (ii) ] If $n$ is sufficiently large, then $P_3$ is $K_{r+1}$-Tur\'{a}n
good.
\end{enumerate}
\end{thm}

Note that in~\cite{GerbnerPalmer2020}, Gerbner and Palmer provided a proof of
part (i) of Theorem~\ref{main_p4_thm}. Part (ii) is an entirely new result. We will
re-prove part (i) in the language of flag algebras, since we will require this
proof to obtain part (ii).

In~\cite{GERBNER2019103001}, Gerbner and Palmer proved that for two graphs $T$ and $F$, where $\chi(F) = r$, 
\[ \ext(n,T,F) \leq \ext(n,T,K_r) + o(n^{|T|}).\]
Combined with Theorem~\ref{main_p4_thm}, their theorem implies the following corollary.

\begin{cor}
For any graph $F$ with chromatic number $r \geq 3$, 
\[ \ext(n,P_3,F) = \OPT_r(P_3)\binom{n}{4} + o(n^4).\]
\end{cor}

In the remainder of this section, we establish the conventions used
thorough the paper, reference a few well-known results that
will be of use throughout the proof, and then provide a brief introduction to the
flag algebra method. Section~\ref{sec:part_i_proof} contains the flag algebra
calculations we use to establish part (i) of Theorem~\ref{main_p4_thm}. In
Section~\ref{sec:stability} we establish a stability result, proving that
near-extremal graphs have small edit distance from the Tur\'an graph. Then in
Section~\ref{sec:exact_result} we use that stability argument to show that the
Tur\'an graph is optimal for large enough $n$. We conclude in
Section~\ref{sec:conclusion} with some thoughts on what this result means for
Gerbner and Palmer's conjecture for general paths $P_\ell$.

\subsection{Background and Conventions}

We use $P_\ell$ to denote the path graph with $\ell$ edges and $\ell+1$
vertices. If a copy of $P_3$ in $G$ is defined by the edges $wx$, $xy$ and $yz$,
then we will use $wxyz$ to denote it. Note that a set of four
vertices in $G$ will frequently give multiple distinct copies of $P_3$. We use $wxyz$ for
that specific ordering. 

\begin{figure}[H]
    \centering
    \begin{tikzpicture}
    
    \node[vtx, label = below:$w$] at (0,0) (w) {};
    \node[vtx, label = below:$x$] at (1,0) (x) {};
    \node[vtx, label = below:$y$] at (2,0) (y) {};
    \node[vtx, label = below:$z$] at (3,0) (z) {};
    
    \draw[gedge] (w)--(x) (x)--(y) (y)--(z);
    
    \end{tikzpicture}
    \caption{The path $wxyz$}
\end{figure}
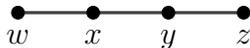

We will need the following corollary of Theorem~\ref{thm_zykov_bound}:
\begin{cor}\label{cor_zykov_bound}
Let $G$ be a $K_{r+1}$-free graph on $n$ vertices. Then 
\[\nu(K_4,G) \leq \frac{r^3 - 6r^2 + 11r - 6}{r^3}\binom{n}{4} + o(n^4)\]
\end{cor}
\begin{proof}
In the Tur\'{a}n graph $T_r(n)$, any set of four vertices inducing a copy of $K_4$ must come from four different partite sets. Thus there are
\[ \binom{r}{4} \cdot \frac{n^4}{r^4} + o(n^4) \]
copies of $K_4$ in $T_r(n)$. The claim immediately follows. 
\end{proof}

We will also need the following lemma from folklore characterizing multipartite graphs:

\begin{lemma}\label{lem:multipartite_characterization}
Define the \emph{co-cherry} $\overline{P_2}$ to be the unique graph on three
vertices with one edge. Then $G$ is a complete multipartite graph if and only if
it does not contain the co-cherry as an induced subgraph.
\begin{figure}[H]
    \centering
    \begin{tikzpicture}
    
    \node[vtx] at (-1,0) {};
    \node[vtx] at (0,0.5) (a) {};
    \node[vtx] at (0,-0.5) (b) {};
    
    \draw[gedge] (a)--(b);
    
    \end{tikzpicture}
    \caption{The co-cherry}
\end{figure}
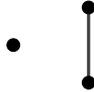
\end{lemma}

\begin{proof}
First, assume $G$ is a complete multipartite graph and let $x,y,z \in V(G)$ such
that $x$ is adjacent to $y$ but $z$ is not adjacent to $y$. As $G$ is complete
multipartite, the only way $z$ is not adjacent to $y$ is if they are in the same
vertex class. As $x$ is adjacent to $y$, it must be in a different vertex class.
Thus $x$ and $z$ do not share a vertex class and are adjacent, so $G[\{x,y,z\}]$
does not span a co-cherry.

Now let $G$ be a graph that does not contain the co-cherry as an induced
subgraph. Define a relation on $V(G)$ by $x\sim y$ if $x$ is not
adjacent to $y$. As $G$ is simple, this relation is reflexive and symmetric, and if $x$ is
not adjacent to $y$ and $y$ is not adjacent to $z$, then $x$ cannot be adjacent
to $z$, as that would form an induced co-cherry, so the relation is transitive
as well. Thefore this equivalence relation partitions the vertices of $G$ into
classes which contain no internal edges. Furthermore, two vertices from
different classes are by definition adjacent and thus every edge between
vertex classes is present. We conclude $G$ is complete multipartite.
\end{proof}

\subsection{The Flag Algebra Method}

Flag algebras were introduced by Razoborov~\cite{Razborov} as a tool to
computationally solve problems in extremal combinatorics. In this section, we
will introduce some of the main ideas necessary for our proof. For a complete
overview see \cite{Razborov}. Flag algebras have been applied to study a
variety of extremal problems on graphs~\cite{flag-max-C5, flag-min-C5,
flag-grzesik, flag-hatami, flag-razborov-triangles} and
hypergraphs~\cite{flag-hypergraphs-Falgas-Ravry-Vaughan, flag-hyper-glebov,
flag-hyper-pikhurko}, as well as oriented graphs~\cite{flag-directed-paths,
flag-hladky-kral}. These only represent a handful of the many results in
combinatorics which were obtained using flag algebras.

A \emph{type} $\sigma$ is a graph labelled by $[k]$. An \emph{embedding} of
$\sigma$ into a graph $F$ is an injective map $\theta: [k] \to V(F)$ so that
$\im(\theta)$ is isomorphic to $\sigma$. A \emph{$\sigma$-flag} $(F,\theta)$ is a
graph $F$ together with an embedding $\theta$ of $\sigma$ into $V(F)$. We will let
$\mathcal{F}^{\sigma}$ denote the set of all $\sigma$-flags up to
isomorphism and $\mathcal{F}^{\sigma}_n$ denote the associated subset
containing all $\sigma$-flags on $n$ vertices. If $\sigma$ is the empty graph,
then we will drop it from the notation and simply use $\mathcal{F}$ to denote the set of all graphs, or $\mathcal{F}_n$ to denote the set of all graphs on $n$ vertices. As an example, if
$\sigma^*$ is the following labelled graph on two vertices,
\[ \sigma^* = \ftwb \]
then
\[ \mathcal{F}^{\sigma^*}_3 = \left\{\ftrf,\ftrc,\ftrd,\ftre \right\}. \]

For a type $\sigma$ labelled by $[k]$, two $\sigma$-flags $(H,\theta_1)$ and $(G,\theta_2)$, and a set $X_1$ of size $|V(H)| - k$ selected uniformly at random from $V(G) \setminus \im(\theta_2)$, $P((H,\theta_1),(G,\theta_2))$ is the probability that $X_1 \cup \im(\theta_2)$ is isomorphic to $(H,\theta_1)$.  For completeness, if $|V(G)| < |V(H)|$, then we let $P(H,G) = 0$. If $\sigma$ is the empty graph, then we will write $P(H,G)$ to mean $P((H,\theta_1),(G,\theta_2))$. In this case, the definition of $P(H,G)$ coincides with the standard notion of induced density. Using the same type $\sigma^*$ from the previous example:
\[ \text{If } H = \ftrc \text{ and } G = \ffiveone, \text{ then } P((H,\theta_1),(G,\theta_2)) = \frac{1}{3}.\]

Now suppose that $(J,\theta_3)$ is another $\sigma$-flag. Let $X_1,X_2 \subseteq V(G)$ be two disjoint sets of size $|V(H)| - k$ and $|V(J)| - k$, respectively, selected uniformly at random from $V(G)\setminus \im(\theta_2)$. Then $P((H,\theta_1),(J,\theta_3);(G,\theta_2))$ is the probability that $X_1 \cup \im(\theta_2)$ is isomorphic to $H$ and $X_2 \cup \im(\theta_2)$ is isomorphic to $J$. Once again, if $\sigma$ is empty, then we write $P(H,J;G)$ in place of $((H,\theta_1),(J,\theta_3);(G,\theta_2))$. Equation~\ref{flag_product_O(n)} follows from the definition of $P((H,\theta_1),(J,\theta_3);(G,\theta_2))$.
\begin{equation}\label{flag_product_O(n)}
|P((H,\theta_1),(J,\theta_3);(G,\theta_2)) - P((H,\theta_1),(G,\theta_2))\cdot P((J,\theta_3),(G,\theta_2))| \leq O(|V(G)|^{-1})
\end{equation}
Thus, as the size of $G$ tends toward infinity, we can assume that we select $X_1$ and $X_2$ independently.

Let $\mathbb{R}\mathcal{F}^{\sigma}$ be the set of all finite formal linear combinations of elements from $\mathcal{F}^{\sigma}$. For a given type $\sigma$, let $\mathcal{K^{\sigma}}$ denote the linear subspace of $\mathbb{R}\mathcal{F^{\sigma}}$ generated by all elements of the form 
\[ 
F - \sum\limits_{(H,\theta_2) \in \mathcal{F}_n^{\sigma}} P((F,\theta_1),(H,\theta_2)) \cdot (H,\theta_2) 
\]
where $|V(F)| < n$. Razborov showed that there exists an algebra $\mathcal{A}^{\sigma} = \mathbb{R}\mathcal{F}^{\sigma}/\mathcal{K}^{\sigma}$ with well defined addition and multiplication. Addition is defined in the natural way by adding coefficients. For example, if $F_1,F_2 \in \mathcal{A}^{\sigma^*}$ such that
\[ F_1 = 2 \cdot \ftrd + \ftre \text{ and } F_2 = \ftrc - \ftrd,\]
then
\[ F_1 + F_2 = \ftrc + \ftrd + \ftre.\]
For a fixed type $\sigma$ of size $k$, if $(F_1,\theta_1)$ and $(F_2,\theta_2)$ are two elements in $\mathcal{F}^{\sigma}$ such that 
\[|V(F_1)| + |V(F_2)| - k = n,\] then the product of $F_1$ and $F_2$ is defined as 
\[ (F_1,\theta_1) \cdot (F_2,\theta_2) = \sum\limits_{(H,\theta_3) \in \mathcal{F}^{\sigma}_n} P((F_1,\theta_1),(F_2,\theta_2);(H,\theta_3)) \cdot (H,\theta_3).\]
For example, if
\[(F_1,\theta_1) = \ftrc \text { and } (F_2,\theta_2) = \ftrd \]
then, 
\[ (F_1,\theta_1) \times (F_1,\theta_2) = \frac{1}{2} \cdot \figsfourlabelone + \frac{1}{2} \cdot \figsfourlabeltwo.\]
Observe that the set $\mathcal{F}_4^{\sigma^*}$ contains more than just the two graphs pictured in the previous equation, but in all of these other graphs, $P((F_1,\theta_1),(F_2,\theta_2);(H,\theta_3)) = 0$. Multiplication in $\mathcal{A}^{\sigma}$ is defined as an extension of multiplication in $\mathcal{F}^{\sigma}$.

A sequence of graphs $(G_n)_{n \geq 1}$, where $|V(G_n)| = n$, is said to be \emph{convergent} if for every finite graph $H$, the limit $\lim\limits_{n \to \infty} P(H,G_n)$ exists. Let $\Hom^+(\mathcal{A}^{\sigma},\mathbb{R})$ denote the set of all homomorphisms from $\mathcal{A}^{\sigma}$ to $\mathbb{R}$ such that $\phi(F) \geq 0$ for each element $F \in \mathcal{F}^{\sigma}$. Razborov showed that each function $\phi \in \Hom^+(\mathcal{A}^{\sigma},\mathbb{R})$ corresponds to some convergent graph sequence $(G_n)_{n \geq 1}$. That is, the values of $\phi$ correspond to the limits of induced densities in $(G_n)_{n \geq 1}$. It is often more intuitive to think of addition and multiplication operations in $\mathcal{A}^{\sigma}$ as representing induced densities of subgraphs in some very large graph $G_{n_0}$ with an error term $O(n_0^{-1})$. 

For each type $\sigma$ labelled by $[k]$, Razborov also defined a function $\llbracket \cdot \rrbracket_{\sigma}: \mathbb{R}\mathcal{F}^{\sigma} \to \mathbb{R}\mathcal{F}$, which we will refer to as the \emph{unlabelling operator}. For a $\sigma$-flag $(F,\theta)$, let $q_{\sigma}(F)$ denote the probability that $(F,\theta')$ is isomorphic to $F$, where $\theta ': V(F) \to [k]$ is a randomly chosen injective mapping. Let $F'$ denote the graph isomorphic to $F$ when ignoring labels. Then
\[ \llbracket F \rrbracket_{\sigma} = q_{\theta}(F)F'.\]
As an example,
\[\text{If, } F = \figsfourlabeltwo, \text{ then }  \llbracket F \rrbracket_{\sigma} = \frac{4}{6} \cdot  \ffrg.\]
Finally, it can be shown using the Cauchy-Schwarz inequality that if $\alpha \in \mathcal{A}^{\sigma}$ is some expression and $\phi \in \text{Hom}^+(\mathcal{A},\mathbb{R})$, then 
\begin{equation}\label{eq: C-S for flags}
   \phi\left(\llbracket \alpha^2 \rrbracket_{\sigma}\right) \geq 0. 
\end{equation}

\section{Theorem~\ref{main_p4_thm} (i)}\label{sec:part_i_proof}

First we will prove a lower bound by counting the number of $P_3$ subgraphs in the Tur\'{a}n graph. After that, the remainder of the section will be devoted to proving the upper bound using flag algebras. 

\begin{lemma}\label{lem:opt_lower_bound}
For all $r \ge 3$,
\[ 12\left(\frac{r-1}{r}\right)^3 \leq \OPT_r(P_3) \]
\end{lemma}

\begin{proof} 
We begin by counting the paths of length three in the Tur\'an graph $T_r(n)$. To
do so, we will first choose the central edge of the path and then select two
additional vertices and describe how to attach them to the central edge.

As the Tur\'an graph is multipartite, the central edge must fall between two
of the $r$ vertex classes. Assume for the moment that $n$ is divisible by $r$.
Then there are $\binom{r}{2}(\frac{n}{r})^2$ choices for the central edge: first
choose two vertex classes and select a vertex from each class.

Now we consider two cases. In the first case, the $P_3$ intersects exactly two
of the vertex classes of $T_r(n)$. In this case, as we have already selected the
central edge, the two vertex classes are already specified and we need only
select an additional vertex from each class. These vertices are each adjacent to
a different vertex of our central edge and thus give a unique $P_3$. There are
$(\frac{n}{r}-1)^2$ ways to choose these two vertices.

In the second case, the $P_3$ intersects at least three vertex classes of
$T_r(n)$. (Note that as vertex classes contain no internal edges, the $P_3$ must
contain vertices from more than one vertex class.) We first select this third
vertex, for which there are $n-2(\frac{n}{r})$ choices, and then select a fourth
unique vertex from the remaining $n-3$ options. If the fourth vertex chosen
happens to share a vertex class with either end of the central edge, then there
is a unique $P_3$ containing the four vertices with the given central edge.
Otherwise, there are two ways to connect the third and fourth vertices to the
central edge. However, we also select pairs of vertices of this form twice as
the fourth vertex we selected was an eligible choice when we selected the third
in this case. Thus either way, this method produces
\[ \left(n-2\left(\frac{n}{r}\right)\right)(n-3) \]
unique copies of $P_3$.

Putting all of our counts together, for all $r \geq 4$, 
\[
\nu(P_3,T_r(n)) = \binom{r}{2}\left(\frac{n}{r}\right)^2\left(
\left(\frac{n}{r} - 1\right)^2 + \left(n - 2\frac{n}{r} \right)(n-3) \right) + o(n^4),
\]
where the error terms accounts for the cases that $n$ is not divisible by $r$.
Factoring out leading leading terms  gives
\[ \nu(P_3,T_r(n)) = n^4 \cdot \frac{1}{2}\left(1 - \frac{1}{r}\right)\left(
\left( \frac{1}{r} - \frac{1}{n}\right)^2 + \left(1 - \frac{2}{r}\right)\left(1
- \frac{3}{n}\right) \right) + o(n^4).\]
As $\binom{n}{4} = \frac{1}{24}n^4 + o(n^4)$, it follows that 
\[ \lim\limits_{n \to \infty} \nu(P_3,T_r(n))\binom{n}{4}^{-1} = \frac{n^4 \cdot \frac{1}{2}\left(1 - \frac{1}{r}\right)\left( \left( \frac{1}{r} - \frac{1}{n}\right)^2 + \left(1 - \frac{2}{r}\right)\left(1 - \frac{3}{n}\right) \right) + o(n^4)}{\frac{1}{24}n^4 + o(n^4)} \] \[ = 12\left(\frac{r-1}{r}\right)^3.\]
Hence, $12\left(\frac{r-1}{r}\right)^3 \leq \OPT_r(P_3)$. 
\end{proof}

We will now prove that $\OPT_r(P_3) \leq 12\left(\frac{r-1}{r}\right)^3$ using the flag algebra method. Unlike many proofs that employ this technique, ours does not require any computer assistance for verification. With that said, this section does require the multiplication and factoring of large polynomials. The authors have included a
link to SageMath code used to verify these calculations in the appendix.

\begin{proof}[Proof of Theorem~\ref{main_p4_thm}(i)]
Let $\mathcal{F}_4 = \{F_i\}_{i = 0}^{10}$ denote the set of all unlabeled
graphs on $4$ vertices up to isomorphism, pictured below.
\begin{figure}[H]
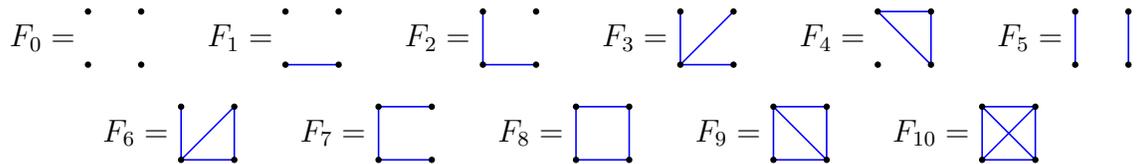

\centering
\[F_0 =  \ffra \hspace{2 em} F_1 = \ffrb \hspace{2 em} F_2 = \ffrc \hspace{2 em} F_3 = \ffrd \hspace{2 em} F_4 =  \ffrj \hspace{2 em} F_5 = \ffrk \] 
\[ F_6 = \ffre \hspace{2 em} F_7 = \ffrf \hspace{2 em} F_8 = \ffrg \hspace{2 em} F_9 = \ffrh \hspace{2 em} F_{10} = \ffri\]
\caption{Enumeration of all graphs in $\mathcal{F}_4$.}
\label{fig:all graphs on four vertices}
\end{figure}
Throughout this section, we will be
working with the induced densities of subgraphs in a convergent
sequence of $K_{r+1}$-free graphs $(G_n)_{n \geq 1}$. In order to simplify
notation we will let $P(F) = \lim\limits_{n \to \infty} P(F,G_n)$ and similarly
$d(F) = \lim\limits_{n \to \infty} d(F,G_n)$. Summing over
all of the graphs on $\mathcal{F}_4$, we observe the following:
\begin{equation}\label{eq_sum_all_in_F4_equal_one}
    \sum\limits_{i = 0}^{10} P(F_i) = 1.
\end{equation}

In order to make expressions like this easier to visualize, we will often use a drawing of $F$ in place of $P(F)$ in our computations. For example, if $(G_n)_{n \geq 1}$ was the sequence of complete graphs on $n$ vertices, then $P(K_4) = \lim\limits_{n \to \infty} P(K_4,G_n) = 1.$ Using a drawing of $K_4$ in order to represent this density, we would write:
\[\ffri = P(K_4) = 1.\]

Fix $r \geq 4$ and let $(G_n)_{n \geq 1}$ be an arbitrary convergent sequence of $K_{r+1}$-free graphs. By the law of total probability, the (non-induced) density of the path $P_3$ can be expressed as the sum of induced densities of graphs on four vertices in the following way, 
\begin{equation}\label{eq_law_total_prob_dens_p3}
d(P_3) = \sum\limits_{i = 0}^{10} P(F_i)\cdot\nu(P_3, F_i).\end{equation}
This expression can be simplified, however, as over half of the graphs in $\mathcal{F}_4$ do not contain a $P_3$ subgraph.
\begin{equation*}
 d(P_3) = \ffrf + 2 \cdot \ffre + 4 \cdot \ffrg + 6 \cdot \ffrh + 12 \cdot \ffri.
\end{equation*}
From Corollary~\ref{cor_zykov_bound} we obtain the following upper bound on $P(K_4)$ in $(G_n)_{n \geq 1}$. 
\begin{equation}\label{eq_zykov flag upper bound}
\ffri \leq \frac{r^3 - 6r^2 + 11r - 6}{r^3}.
\end{equation}
Note that
\begin{align*}
P_0(r) &:= \sum\limits_{i = 0}^{9} \left(\frac{r^3 - 6r^2 + 11r -
6}{r^3}\right) \cdot P(F_i) + \ffri\frac{-6r^2 + 11r - 6}{r^3}\\
	&=  \left(\frac{r^3 - 6r^2 + 11r -
6}{r^3}\right) - \ffri & \text{by \eqref{eq_sum_all_in_F4_equal_one}}\\
	&\ge 0 & \text{by \eqref{eq_zykov flag upper bound}}
\end{align*}

In the following computations, we will use two sets of labeled flags $\mathcal{F}_3^{\sigma_1}$ and $\mathcal{F}_3^{\sigma_2}$, where  
\[ \sigma_1 = \ftwa, \]  
\[\sigma_2 = \ftwb. \]
By the Cauchy-Schwarz inequality, each of the following three expressions is nonnegative for all $r \geq 4$.  

\begin{enumerate}
    \item $P_1(r) = 6 \cdot \left\llbracket \left( (r-1)\ftra - \ftrb \right)^2 \right\rrbracket_{\sigma_1} = $ 
    $$ (6r^2 - 12r + 6) \cdot \ffra + (r^2 - 2r +1) \cdot \ffrb + (1-r) \cdot \ffrc + (3-3r) \cdot \ffrd  + 2 \cdot \ffrg + \ffrh$$
    \item $P_2(r) = 6 \cdot \left\llbracket \left(\ftrc - \ftrd\right)^2 \right\rrbracket_{\sigma_2} = $
    $$ 3\ffrd + \ffrf - \ffre - 4\ffrg $$
    \item $P_3(r) = 6 \cdot \left\llbracket \left( (r-2)\ftrc + (r-2)\ftrd - 2\ftre \right)^2 \right\rrbracket_{\sigma_2} = $
    $$(3r^2 - 12r + 12) \cdot \ffrd + (r^2 - 8r +12) \cdot \ffrf + (r^2 - 6r + 12) \cdot \ffre + (4r^2 - 16r + 16) \cdot \ffrg $$
    $$+ (20 - 8r) \cdot \ffrh + 24 \cdot \ffri $$
\end{enumerate}
Moreover, It can quickly verified that for all $r \geq 4$, the following polynomials are all nonnegative.
\begin{enumerate}
    \item $p_0(r) = \frac{18(r^2 - 2r + 1)}{3r^2 - 11r + 9}$
    \item $p_1(r) = \frac{3r^3 - 10r^2 + 7r}{3r^5 - 11r^4 + 9r^3}$
    \item $p_2(r) = \frac{9r^5 - 32r^4 + 25r^3}{4(3r^5 - 11r^4 + 9r^3)}$
    \item $p_3(r) = \frac{15r^3 - 24r^2 + 7r}{4(3r^5 - 11r^4 + 9r^3)}$
\end{enumerate}
We can add the sum $\sum\limits_{j=0}^{3} p_j(r)P_j(r)$ to Equation~\eqref{eq_law_total_prob_dens_p3} to obtain the following upper bound on $d(P_3)$. 
\begin{equation}\label{eq unsimplified final flag path bound}
    d(P_3) \leq \sum\limits_{i = 0}^{10} P(F_i)\cdot\nu(P_3, F_i) + \sum\limits_{j=0}^{3} p_j(r)P_j(r).
\end{equation}
For each $F_i \in \mathcal{F}_4$, let $C_{F_i}$ denote the coefficient of the graph $F_i$ after combining like-terms in Equation~\eqref{eq unsimplified final flag path bound}.
This gives the following, simplified upper bound on $d(P_3)$.
\[
d(P_3) \leq \sum\limits_{i = 0}^{10} C_{F_i}P(F_i).
\]
Since $\sum\limits_{i = 0}^{10} P(F_i) = 1$, it follows that 
\begin{equation}\label{eq: showing that density is less than max coefficient}
d(P_3) \leq \max\{ C_{F_i} : F_i \in \mathcal{F}_4 \}.
\end{equation}
The following are the exact values of each $C_{F_i}$.

\begin{itemize}
    \item $C_{F_0} = C_{F_3} = C_{F_8} = C_{F_9} = C_{F_{10}} = $ 
    \[ 
    12\left(\frac{r-1}{r}\right)^3
    \]
    \item $C_{F_1} =$
    \[
    \frac{(21r^2 - 97r + 108)(r - 1)^3}{3r^5 - 11r^4 + 9r^3}
    \]
    \item $C_{F_2} =$
    \[ \frac{(18r^3 - 111r^2 + 205r - 108)(r - 1)^2}{3r^5 - 11r^4 + 9r^3}
    \]
    \item $C_{F_4} = C_{F_5} = $
    \[ \frac{18(r - 1)^3(r - 2)(r - 3)}{3r^5 - 11r^4 + 9r^3}
    \]
    \item $C_{F_6} =$
    \[
    \frac{45r^5 - 351r^4 + 1035r^3 - 1389r^2 + 870r - 216}{2(3r^5 - 11r^4 + 9r^3)}
    \]
    \item $C_{F_7} =$
    \[ 
    \frac{(30r^4 - 180r^3 + 371r^2 - 327r + 108)(r - 1)}{3r^5 - 11r^4 + 9r^3}
    \]
\end{itemize}
By examining leading coefficients and factoring, it is clear that for all $r > 1000$,
\begin{equation}\label{eq:factoring to show OPT is what we want}
\max\{ C_{F_i} : F_i \in \mathcal{F}_4 \} = 12\left(\frac{r-1}{r}\right)^3.
\end{equation}
We have provided a link for SageMath code which can be used to verify \eqref{eq:factoring to show OPT is what we want} for $4 \leq r \leq 1000$. 
This fact, together with Equation~\ref{eq: showing that density is less than max coefficient} are enough to show that 
\[ \OPT_r(P_3) \leq 12\left(\frac{r-1}{r}\right)^3.\]
Along with Lemma~\ref{lem:opt_lower_bound}, this completes the proof of Theorem~\ref{main_p4_thm}(i). \end{proof}

\section{Stability}\label{sec:stability}
For two graphs $G$ and $H$ of the same order, the \emph{edit distance} between
$G$ and $H$, denoted $\Dist(G,H)$, is the minimum number of adjacencies
one needs to add or remove in order to change $G$ into a graph isomorphic to
$H$. Our goal in this section is to prove that graphs with $P_3$ density approaching
$\OPT_r(P_3)$ are close in structure to the Tur\'an graph $T_r(n)$.
Specifically, we prove the following lemma:

\begin{lemma}\label{lem:actual_stability}
For every $\varepsilon > 0$, there exists an $n_0$ and $\delta > 0$ such
that for every $K_{r+1}$-free graph $G$ of order $n \geq n_0$, if $d(P_3,G) \geq
\OPT_r(n) - \delta$, then $\Dist(G,T_r(n)) \leq \varepsilon n^2$. 
\end{lemma}

We prepare for the proof of Lemma~\ref{lem:actual_stability} with a collection of
lemmas. Several of these lemmas use the epsilon-delta paradigm, and so in the
interest of legibility we have labelled the lemmas in this section by letter. We
adopt the convention that $\varepsilon_A$, for example, will always refer to the
$\varepsilon$ in Lemma A. The exception to this rule is
Lemma~\ref{lem:actual_stability} which uses unadorned variables.

The first lemma is the Induced Removal Lemma, proved by Alon, Fischer, Krivelevich and
Szegedy~\cite{Alon_efficienttesting_removlemma}.

\begin{lemma}[Lemma A, Induced Removal Lemma]\label{lem:removal_lemma}
Let $\mathcal{F}$ be a set of graphs. For each $\varepsilon_A > 0$, there exist
$\eta_A$ and $\delta_A > 0$ such that for every graph $G$ of order $n \geq
\eta_A$, if $G$ contains at most $\delta_A n^{|V(H)|}$ induced copies of $H$ for every
$H \in \mathcal{F}$, then $G$ can be made $\mathcal{F}$-free by removing or
adding at most $\varepsilon_A n^2$ edges from $G$. 
\end{lemma}

We define the set $T$ to contain all of the graphs $F \in
\mathcal{F}_4$ for which $c_F = \OPT_r(P_3)$ in the proof of Theorem~\ref{main_p4_thm}. 
\[
T = \left\{ \hspace{ 1 em} \ffra, \hspace{2 em} \ffrd, \hspace{2 em} \ffrg, \hspace{2 em} \ffrh, \hspace{2 em} \ffri \hspace{1 em }\right\}.
\]
The following is a restatement of Lemma 2.4.3 appearing in~\cite{baberthesis}. For completeness, we will provide a short proof.
\begin{lemma}{\cite{baberthesis}}\label{lem:baber thesis}
Let $(G_n)_{n \geq 1}$ be a sequence of $K_{r+1}$-free graphs such that 
\[ 
\lim\limits_{n \to \infty} d(P_3,G_n) = \lim\limits_{n \to \infty} \sum\limits_{i=0}^{10} C_{F_i} \cdot P(F_i,G_n) =  \OPT_r(P_3),
\]
where $F_i \in \mathcal{F}_4$ for all $i = 0,\dots, 10$. Then for all $F \in \mathcal{F}_4$, $\lim\limits_{n \to \infty} P(F,G_n) > 0$ implies that $F \in T$.
\end{lemma}

\begin{proof}
Let $\mathcal{F}_4^*$ denote the set of graphs $F$ in $\mathcal{F}_4$ for which $\lim\limits_{n \to \infty} P(F,G_n) > 0$. Then $
\lim\limits_{n \to \infty}\sum_{F \in \mathcal{F}_4^*} P(F,G_n) = 1,
$
implying  from Theorem~\ref{main_p4_thm}(i) that
\[ \lim\limits_{n \to \infty}\sum\limits_{F \in \mathcal{F}_4^*} C_F \cdot P(F,G_n) = \OPT_r(P_3).
\]
For each graph $H \in \mathcal{F}_4 \setminus T$, we know from the proof of Theorem~\ref{main_p4_thm}(i) that $C_H < \OPT_{r}(P_3)$. Thus, $H \notin \mathcal{F}_4^*$ as otherwise $\lim\limits_{n \to \infty}\sum_{F \in \mathcal{F}_4^*} C_F \cdot P(F,G_n) < \OPT_r(P_3).$
\end{proof}

Given the fact that only those graphs in $T$ can appear with positive density in the limit of any extremal sequence, we can now prove the following lemma.

\begin{lemma}[Lemma B]\label{lem:tight_graphs}
For each $\varepsilon_B > 0$, there exists a $\eta_B$ and $\delta_B > 0$ such
that any $K_{r+1}$-free graph $G$ of order $n \ge \eta_B$ satisfying $d(P_3,G)
\geq \OPT_r(P_3) - \delta_B$ contains at most $\varepsilon_B n^3$ copies of
$\overline{P_2}$. 
\end{lemma}

\begin{proof}
Suppose that $(G_n)_{n \geq 1}$ is some convergent sequence of $K_{r+1}$-free graphs for which 
\[ 
\lim\limits_{n \to \infty} d(P_3,G_n) = \OPT_r(P_3).
\]
By inspection, none of the graphs in $T$ contain $\overline{P_2}$ as a subgraph. Thus from Lemma~\ref{lem:baber thesis},
\[ \lim\limits_{n \to \infty} d(\overline{P_2},G_n) = 0.\]
This fact immediately implies Lemma~\ref{lem:tight_graphs}.
\end{proof}

Next we prove that among all complete $r$-partite graphs on at least four
vertices, the Tur\'{a}n graph $T_r(n)$ contains the most $P_3$ subgraphs.

\begin{lemma}\label{lem:rpartite_turan_best}
For $n \geq 4$ and $r \geq 4$, if $G$ is any complete $r$-partite graph on $n$
vertices then $\nu(P_3,G) \leq \nu(P_3,T_r(n)).$
\end{lemma}

\begin{proof}
We count the number of $P_3$ in a complete multipartite graph using a similar
approach to that in the proof of Theorem~\ref{lem:opt_lower_bound}. We sum over
each edge and count the number of $P_3$ with that edge as the center. If $e =
xy$ is an edge in the center of $P_3$ with $x$ in vertex class $V_x$ and $y$ in
vertex class $V_y$, let the other edges of the $P_3$ be $wx$ and $yz$. We
classify the $P_3$ into one of four types depending on the location of $w$ and
$z$.
\begin{itemize}
\item There are $(|V_x|-1)(|V_y|-1)$ such $P_3$ with $w \in V_y$ and $z \in
V_x$ as we may not reselect $x$ or $y$.
\item When $w \in V_y$ but $z \notin V_x$, there are $(|V_y|-1)(n-|V_x|-|V_y|)$
choices for the $P_3$ as $z$ falls in some vertex class other than $V_x$ or
$V_y$.
\item Similarly, when $w \notin V_y$ and $z \in V_x$, there are
$(n-|V_x|-|V_y|)(|V_x|-1)$ many such $P_3$.
\item Finally, if $w \notin V_y$ and $z \notin V_x$, then we must take care to
select them uniquely.  Choosing $w$ first and then $z$ gives
$(n-|V_x|-|V_y|)(n-|V_x|-|V_y|-1)$ many such $P_3$.
\end{itemize}
Thus in total, for complete multipartite graphs $G$,
\begin{align*}
\nu(P_3,G) &= \sum_{e=xy} (|V_x|-1)(|V_y|-1) + (|V_y|-1)(n-|V_x|-|V_y|) \quad +\\
	&\qquad\qquad+\quad(n-|V_x|-|V_y|)(|V_x|-1) + (n-|V_x|-|V_y|)(n-|V_x|-|V_y|-1)\\
	&= \sum_{e=xy} (|V_x|-1)(|V_y|-1) + (n-|V_x|-|V_y|)(n-3)
\end{align*}
Now suppose that $G$ has $r$ parts $V_1, \ldots, V_r$. There are $|V_i||V_j|$
edges between parts $V_i$ and $V_j$, each of which contributes the same term in
the sum above. Thus we may also write
\begin{equation}\label{eq:multipartite_p3_count}
\nu(P_3,G) = \sum_{1 \le i < j \le r}|V_i||V_j|\bigl((|V_i|-1)(|V_j|-1) + (n-|V_i|-|V_j|)(n-3)\bigr).
\end{equation}

Let $G$ be a complete $r$-partite graph on $n$ vertices with parts $V_1,
\ldots, V_r$ such that $|V_1| \ge |V_2|+2$. If $G$ has no edges but at least
four vertices, it cannot be extremal, so assume $G$ contains at least one edge.
Define $G'$ to be the complete multipartite graph on $n$ vertices with parts
$V_1', V_2', \ldots, V_r'$ where $|V_1'| = |V_1|-1$, $|V_2'| = |V_2|+1$, and
$|V_i'| = |V_i|$ for $i \ge 3$.

After straightforward, if tedious, calculation, we use
(\ref{eq:multipartite_p3_count}) to see $\nu(P_3,G')- \nu(P_3,G) = \Delta_{P_3}$ where
\[ \Delta_{P_3} = (|V_1|-|V_2|-1)\bigl((n-|V_1|-|V_2|)(n-3) + 2(|V_1|-1)|V_2|+
\sum_{j=3}^r |V_j|(n-2-|V_j|)\bigr) \]

Note that by assumption $|V_1| \ge |V_2|+2$ and $n \ge 4$, that $n \ge |V_1|+|V_2|$
as $V_1, V_2 \subseteq V(G)$ and that $n - 2 \ge |V_j|$ for $j \ge 3$ because $|V_j|
\le n - |V_1|$ and $V_1$ must have at least two vertices to satisfy $|V_1| \ge
|V_2|+2$. Thus $|V_1|-|V_2|-1$ is strictly positive and $(n-|V_1|-|V_2|)(n-3)$,
$2(|V_1|-1)|V_2|$, and $\sum_{j=3}^r |V_j|(n-2-|V_j|)$ are each nonnegative. If
$\Delta_{P_3} = 0$, then each term must be exactly zero. This means $n = |V_1| +
|V_2|$ and $|V_2| = 0$. But then $n = |V_1|$, so all of the vertices of $G$ are
in one part which contradicts that $G$ has at least one edge. We conclude
$\Delta_{P_3} > 0$ and thus $G'$ contains more $P_3$ than $G$.

Thus we see $G$ was not extremal and therefore the Tur\'an graph, the unique
complete $r$-partite graph in which no pair of vertex classes differs in size by
more than one, is the complete $r$-partite graph with the greatest number of $P_3$.
\end{proof}

In the next lemma, we prove that if $G$ has large $P_3$-density, it is close in
edit distance to a nearly balanced complete $r$-partite graph.

\begin{lemma}[Lemma C]\label{lem:partite_sizes_mostly_balanced}
For any two independent parameters $\varepsilon_C > 0$ and $\gamma_C
> 0$ there are $\eta_C$ and $\delta_C > 0$ such that if $G$ is a $K_{r+1}$-free graph with
order $n \ge \eta_C$ satisfying $d(P_3,G) > \OPT_r(P_3)-\delta_C$, then there is
a complete $r$-partite graph $G'$ with parts $X_1, \ldots, X_r$ satisfying
$\Dist(G,G') \le \gamma_C n^2$ and, for each $1 \le i \le r$,
\[ \frac{1-\varepsilon_C}{r}n \le |X_i| \le \frac{1+\varepsilon_C}{r}n. \]
\end{lemma}

\begin{proof}
Let $\varepsilon_C, \gamma_C > 0$ be given. We require a
$\gamma_C' > 0$ but defer its exact definition until later. Take
$\eta_A$ and $\delta_A$ to be as in Lemma~\ref{lem:removal_lemma} so that any
graph $G$ of order $n \ge \eta_A$ containing at most $\delta_A n^3$ copies of
$\overline{P_2}$ can be made $\overline{P_2}$-free by editing at most
$\gamma_C' n^2$ edges. Then take $\eta_B$ and $\delta_B$ to be as in
Lemma~\ref{lem:tight_graphs} so that for any graph $G$ of order $n \ge \eta_B$
which satisfies $d(P_3,G) \ge \OPT_r(P_3)-\delta_B$ contains at most $\delta_A
n^3$ copies of $\overline{P_2}$ (that is, apply Lemma~\ref{lem:tight_graphs}
with $\varepsilon_B = \delta_A$).

Though we are not ready to define them yet, we will ensure $\eta_C \ge
\max(\eta_A, \eta_B)$ and $\delta_C \le \min(\delta_A,\delta_B)$. Let $G$ be a
graph of order $n \ge \eta_C$ satisfying $d(P_3,G) \ge \OPT_r(n)-\delta_C$. By
Lemma~\ref{lem:tight_graphs}, $G$ has at most $\delta_An^3$ copies of
$\overline{P_2}$ and thus by Lemma~\ref{lem:removal_lemma} we may edit at most
$\gamma_C' n^2$ edges of $G$ to get a $\overline{P_2}$-free graph, $G'$. It
follows from Lemma~\ref{lem:multipartite_characterization} that $G'$ is a
complete $r$-partite graph as it is both $K_{r+1}$-free and
$\overline{P_2}$-free. Let $X_1,\dots,X_r$ denote the partite sets of $G'$.
We complete the proof by demonstrating these partite sets all have size nearly
$\frac{n}{r}$.

There is a constant $c > 0$ such that each edge removed from $G$ is contained in at most
$cn^2$ copies of $P_3$. (The constant $c$ counts the number of ways to extend an
edge and two other vertices into a copy of $P_3$.) Thus
\[ d(P_3,G') \geq \OPT_r(P_3) - \delta_C - c\gamma_C' \]
as the $P_3$-density of the removed edges is at most
\[ \frac{\gamma_C' n^2\cdot c n^2}{n^{|V(P_3)|}} =
c\gamma_C'. \]

To prove that the partite sets have bounded size, we will show that if they do
not, we may alter $G'$ to increase its $P_3$ density beyond $\OPT_r(P_3)$.
As $\OPT_r(P_3)$ is, by definition, a limit, we can, for large enough $\eta_C$,
get upper bounds on the $P_3$-density of such graphs that are as close to
$\OPT_r(P_3)$ as necessary to arrive at a contradiction.

We require a partial result from the proof of
Lemma~\ref{lem:rpartite_turan_best}. Recall that when moving one vertex from
vertex class $V_1$ to vertex class $V_2$ the change in the number of $P_3$
subgraphs was

\[ \Delta_{P_3} = (|V_1|-|V_2|-1)\bigl((n-|V_1|-|V_2|)(n-3) + 2(|V_1|-1)|V_2|+
\sum_{j=3}^r |V_j|(n-2-|V_j|)\bigr) \]

Assume first that there is a partite set that is too large. Specifically,
assume, without loss of generality, that $|X_1| > \frac{1+\varepsilon_C}{r}n$. We
consider two cases.

First, assume
\[ \frac{1+\varepsilon_C}{r}n < |X_1| \le \frac{n}{2}. \]
There must be a partite set of $G'$, say $X_2$, that satisfies $|X_2| <
\frac{n}{r}$; if not,
\[ n = \sum_{i=1}^r |X_i| \ge \frac{1+\varepsilon_C}{r}n + (r-1)\frac{n}{r} = n +
\frac{\varepsilon_C}{r}n \]
is a contradiction. Consider the process of moving one vertex from
$X_1$ to $X_2$ repeated $\frac{\varepsilon_C}{3r}n$ times. At each step of this
process,
\[ |X_1| - |X_2| > \left(\frac{1+\varepsilon_C}{r}n -
\frac{\varepsilon_C}{3r}n\right) - \left(\frac{1}{r}n +
\frac{\varepsilon_C}{3r}n\right) = \frac{\varepsilon_C}{3r} n. \]
We take $\eta_C$ large enough that this value is always at least $2$ so that
number of $P_3$ subgraphs increases at every step. In
particular, as $|X_1|+|X_2|$ stays constant and
\[ |X_1|+|X_2| < \frac{n}{2} + \frac{n}{r} \le \frac{3}{4}n \]
we have
\[ \Delta_{P_3} \ge (|X_1|-|X_2|-1)((n-|X_1|-|X_2|)(n-3)) \ge
\left(\frac{\varepsilon}{3r}n-1\right)\left((n-\tfrac{3}{4}n)(n-3)\right). \]
Take $\eta_C$ large enough so that $n \ge \eta_C$ implies $n-3 \ge \frac{n}{2}$
and $\frac{\varepsilon_C}{3r}n-1 \ge \frac{\varepsilon_C}{4r}n$, giving
\[ \Delta_{P_3} \ge \frac{\varepsilon_C}{4r}n \cdot \frac{n}{4} \cdot \frac{n}{2}
= \frac{\varepsilon_C}{32r}n^3. \]
Now, as we repeat this process $\frac{\varepsilon_C}{3r} n$ times, the total
increase in the number of copies of $P_3$ is at least
\[ \frac{\varepsilon_C}{3r}n \cdot \frac{\varepsilon_C}{32r}n^3 =
\frac{\varepsilon_C^2}{96r^2}n^4. \]
As $|V(P_3)| = 4$, this increases the $P_3$ density of $G'$ by at least
$\frac{\varepsilon_C^2}{96r^2}$. By choosing $\delta_C$ and $\gamma_C'$ such
that
\[ \delta_C + c\gamma_C' < \frac{\varepsilon_C^2}{96r^2} \]
we arrive at a graph $G''$ with
\[ d(P_3,G'') \ge d(P_3,G') + \frac{\varepsilon_C^2}{96r^2} \ge \OPT_r(P_3)
-\delta_C -c\gamma_C' +\frac{\varepsilon_C^2}{96r^2} > \OPT_r(P_3),\]
a contradiction for large enough $\eta_C$.

Otherwise we have $|X_1| \ge \frac{n}{2}$. We wish to use a similar approach to
the first case, but we must assure that the lower bound on $\Delta_{P_3}$ is
cubic in $n$ at each step of the process. Note that for $2 \le i \le r$, we must have
$|X_i| \le \frac{1}{2(r-1)}n \le \frac{1}{6}n$ (recall $r \ge 4$). We start by
moving $\frac{n}{12}$ vertices from $X_1$ to $X_2$. These moves increase the
number of copies of $P_3$, but we disregard those increases. Then we have $|X_1|
> \frac{n}{2}-\frac{n}{12} = \frac{5}{12}n$ and
\[ \frac{n}{12} \le |X_2| \le \frac{n}{6} + \frac{n}{12} = \frac{3}{12}n. \]
Starting from this modified graph we can move $\frac{n}{24}$ additional
vertices from $X_1$ to $X_2$. For each such move, we have
\[ (|X_1|-|X_2|-1) \ge \left(\frac{5}{12}n-\frac{1}{24}n\right) -
\left(\frac{3}{12}n+\frac{1}{24}n\right) - 1 = \frac{n}{12} - 1 \ge \frac{n}{13} \]
by choosing $\eta_C$ large enough, and
\[ 2(|X_1|-1)|X_2| \ge 2\left(\frac{5}{12}n-1\right)\left(\frac{n}{12}\right) >
\frac{n^2}{18}, \]
again with $\eta_C$ large enough. Thus
\[ \Delta_{P_3} \ge \frac{n}{13}\cdot \frac{n^2}{18} = \frac{n^3}{234} \]
and repeating this process $\frac{n}{24}$ times increases the total number of
$P_3$ subgraphs by at least $\frac{n^4}{5616}$, increasing the $P_3$ density of $G'$ by
$\frac{1}{5616}$. By taking $\delta_C + c\gamma_C' < \frac{1}{5616}$ we again
get a graph with $P_3$ density larger than the optimal density, a contradiction
when $\eta_C$ is sufficiently large.

Finally, we now assume for contradiction that $|X_1| <
\frac{1-\varepsilon_C}{r}n$. If $|X_1| < \frac{1-(r-1)\varepsilon_C}{r}n$, then
there must be another partite set $X_i$ with $|X_i| > \frac{1+\varepsilon_C}{r}n$
as otherwise
\[ n = \sum_{i=1}^r |X_i| < \frac{1-(r-1)\varepsilon_C}{r}n +
(r-1)\frac{1+\varepsilon_C}{r}n = n\]
is a contradiction. As we have already handled cases with a too large part, we
may assume
\[ \frac{1-(r-1)\varepsilon_C}{r}n \le |X_1| < \frac{1-\varepsilon_C}{r}n. \]
There must be a partite set $X_i$ with $|X_i| > \frac{n}{r}$, again because
otherwise the parts combined cannot contain $n$ vertices. Then we move a vertex
from $X_i$ to $X_1$ and repeat the move $\frac{\varepsilon}{3r}n$ times. Then as
before at every step of the process
\[ |X_i|-|X_1| \ge \frac{1-\varepsilon}{3r}n > 0 \]
and, using very rough bounds,
\[ |X_1|+|X_i| \le \frac{1-\varepsilon_C}{r}n + \frac{1+\varepsilon_C}{r}n <
\frac{n}{r}+\frac{n}{2} \le \frac{3}{4}n. \]
Therefore this process also increases the $P_3$ density of $G'$ by at least
$\frac{\varepsilon_C^2}{96r^2}$, a contradiction for $\delta_C$ small enough. We
conclude each partite set $X_1, \ldots, X_r$ must be within the specified
bounds.

For completeness, we explicitly specify our choices of $\eta_C, \delta_C$, and
$\gamma_C'$. We set
\begin{align*}
\delta_C &= \min\left(\delta_A, \delta_B, \frac{1}{20000},
\frac{\varepsilon_C^2}{200r^2}\right)\\
\gamma_C' &= \min\left(\gamma_C, \frac{1}{20000c},
\frac{\varepsilon_C^2}{200cr^2}\right)\\
\eta_C &\ge \max\left(\eta_A, \eta_B, \frac{12r}{\varepsilon_C}, 144\right)
\end{align*}
where $\eta_C$ is also large enough to guarantee all graphs of this form are
sufficiently close to $\OPT_r(P_3)$.

These choices assure that we can combine Lemmas~\ref{lem:removal_lemma} and \ref{lem:tight_graphs}
to produce a $G'$ with $\Dist(G,G') \le \gamma_C' n^2 \le \gamma_C n^2$ also
that
\[ \delta_C + c\gamma_C' \le \frac{1}{20000} + \frac{c}{20000c} =
\frac{1}{10000} < \frac{1}{5616} \]
and
\[ \delta_C + c\gamma_C' \le \frac{\varepsilon_C^2}{200r^2} +
\frac{c\varepsilon_C^2}{200cr^2} = \frac{\varepsilon_C^2}{100r^2} <
\frac{\varepsilon_C^2}{96r^2}, \]
as well as the bounds we use on $n$, all hold.
\end{proof}

We are now ready to prove Lemma~\ref{lem:actual_stability}.

\begin{proof}[Proof of Lemma~\ref{lem:actual_stability}]
Let $\varepsilon > 0$ be given. Set $n_0 = \eta_C$ and $\delta = \delta_C$ from
Lemma~\ref{lem:partite_sizes_mostly_balanced} with $\gamma_C = \varepsilon/2$
and $\varepsilon_C = \varepsilon/2r$. Then given a graph $G$ of order $n \ge
n_0$ that satisfies $d(P_3, G) \ge \OPT_r(P_3)-\delta$, we get a complete
$r$-partite graph $G'$ satisfying $\Dist(G,G') \le \frac{\varepsilon}{2}n^2$ and
with parts $X_1, \ldots, X_r$ satisfying
\[ \frac{1-\frac{\varepsilon}{2r}}{r}n \le |X_i| \le
\frac{1+\frac{\varepsilon}{2r}}{r}n. \]
We claim $\Dist(G',T_r(n)) \le \frac{\varepsilon}{2}n^2$. From each of the $r$
parts, at most $\frac{\varepsilon}{2r}n$ vertices must be added to or removed
from that part. Thus in total, $\frac{\varepsilon}{2}n$ vertices are altered.
Each vertex requires changing at most $n$ adjacencies, so the total edit
distance is bounded above by $\frac{\varepsilon}{2}n^2$.

Finally, by first making the at most $\frac{\varepsilon}{2}n^2$ edits to change
$G$ into $G'$ and then making the at most $\frac{\varepsilon}{2}n^2$ edits to
change $G'$ into $T_r(n)$, we have demonstrated $\Dist(G,T_r(n)) \le \varepsilon
n^2$, completing the proof.
\end{proof}

\section{Exact Result}\label{sec:exact_result}

In this section we will prove Theorem~\ref{main_p4_thm}(ii). We now know that
for large enough $n$, if $G$ is an $n$-vertex $K_{r+1}$-free graph that is close
to being extremal, then $G$ is close in edit-distance to $T_r(n)$. As we will
show in this section, the process of adding or removing the necessary edges in
order to transform $G$ into $T_r(n)$ must increase the number of
$P_3$-subgraphs in $G$. First we need the following proposition, which
shows that in any extremal graph each pair of vertices must be contained in
approximately the same number of $P_3$-subgraphs. We define $\nu_G(v,T)$ as the number of (not necessarily induced) subgraphs of a graph $G$
isomorphic to $T$ containing $v$.

\begin{prop}\label{prop_lower_bound_path_vertex}
Fix $r \geq 4$. Then there exists an $n_0 = n_0(r)$ such that if $G$ a $K_{r+1}$-free graph on $n \geq n_0$ vertices for which $\nu(P_3,G) = \ext(n,P_3,K_{r+1})$, then for every vertex $v \in V(G)$
\[ 
\nu_G(v,P_3) \geq \left( \OPT_r(P_3) - \frac{1}{r^{10}} \right)\binom{n-1}{3} - \frac{1}{r^4}n^3 . 
\]
\end{prop}

\begin{proof}
From the proof of Theorem~\ref{main_p4_thm}(i), there must exist some $n_0$ such that 
\[ 
\nu(P_3,G) \geq \left( \OPT_r(P_3) - \frac{1}{r^{10}} \right)\binom{n}{4} \]
for every extremal graph $G$ on $n \geq n_0$ vertices. Suppose that $G$ is such
a graph on $n \geq \max \{n_0,2r^4\}$ vertices. We count the copies of $P_3$ in $G$
in two ways to see
\[ 
\sum\limits_{v \in V(G)} \nu_G(v,P_3) = 4 \nu(P_3,G) \geq 4 \left( \OPT_r(P_3) - \frac{1}{r^{10}} \right)\binom{n}{4}.
\]
Thus, by averaging there must exist some vertex $u \in V(G)$ for which
\[ 
\nu_G(u,P_3) \geq \left( \OPT_r(P_3) - \frac{1}{r^{10}} \right)\binom{n-1}{3}. 
\]
Suppose for contradiction that for some vertex $v \in V(G)$,
\[ \nu_G(v,P_3) < \left( \OPT_r(P_3) - \frac{1}{r^{10}} \right)\binom{n-1}{3} - \frac{1}{r^4}n^3.\]

Let $G'$ be the graph obtained from $G$ by deleting $v$ and replacing it with a
vertex $u'$ so that $N(u') = N(u)$. We claim that $G'$ is $K_{r+1}$-free.
Suppose for contradiction that it is not. Then $u'$ must be contained in every
copy of $K_{r+1}$ in $G'$. As $u$ is not adjacent to $u'$, none of these
$K_{r+1}$ contain $u$. However, since $N(u) = N(u')$, this implies that we
can replace $u'$ with $u$ in each $(r+1)$-clique. Since $V(G') - \{u'\} = V(G) -
\{v\}$, this implies the existence of an $(r+1)$-clique in $G$, which is a
contradiciton.

Let $\nu_G(u,v,P_3)$ denote the number of $P_3$ subgraphs containing both $u$ and $v$ in $G$. Then since $\nu_G(u',P_3) = \nu_G(u,P_3)$, we have added at least $\nu_G(u,P_3) -  \nu_G(u,v,P_3)$ subgraphs and removed at most $\nu_G(v,P_3)$ subgraphs. Hence,
\[
\nu(P_3,G') = \nu(P_3,G) + \nu_G(u,P_3) - \nu_G(u,v,P_3) - \nu_G(v,P_3)
\]
Since $\nu_G(u,v,P_3) \leq 2n^2$,
\[ 
\nu(P_3,G') > \nu(P_3,G) + \frac{1}{r^4}n^3 - 2n^2.
\]
By assumption, $\frac{1}{r^4}n^3 - 2n^2 > 0$. This would imply that $\nu(P_3,G') > \nu(P_3,G)$ which contradicts the assumption that $G$ was extremal, completing the proof.
\end{proof}

We will also require the following proposition much later in the proof of Theorem~\ref{main_p4_thm}(ii), where we will provide more explanation of why it is required. For completeness, we will state it here.  

\begin{prop}\label{prop:n_0 bounds for the type 2 and 1 claims}
For all integers $r \geq 4$, there exists an $n_0 = n_0(r)$ such that for all $n \geq n_0$, 
\begin{itemize}
\item[(i)] $\OPT_r(P_3)\binom{n-1}{3} - \delta_1(r)\frac{n^3}{6} \geq \left(\frac{9}{r} - \frac{39}{2r^2}\right)n^3$,
where 
\[ \delta_1(r) = 12 - \frac{45}{r} + \frac{111}{2r^2} - \frac{27}{2r^3} - \frac{21}{r^4} + \frac{24}{r^5} + \frac{3}{2r^6} - \frac{3}{2r^7}.\]
\item[(ii)] $\OPT_r(P_3)\binom{n-1}{3} - \delta_2(r)\frac{n^3}{6} \geq \left(\frac{18}{r} - \frac{42}{r^2} - \frac{12}{r^3} \right)n^3,$
where 
\[ \delta_2(r) = 12 - \frac{54}{r} + \frac{78}{r^2} - \frac{96}{r^4}  + \frac{72}{r^5} + \frac{24}{r^6} - \frac{24}{r^7}\]
\end{itemize}
\end{prop}

\begin{proof}
Part (i) immediately follows from the inequality below, which is true for all $r \geq 4$.
\[ 
\OPT_r(P_3) - \delta_1(r) = \frac{9}{r} - \frac{39}{2r^2} + \frac{3}{2r^3} + \frac{21}{r^4} - \frac{24}{r^5}  - \frac{3}{2r^6} + \frac{3}{2r^7} > \frac{9}{r} - \frac{39}{2r^2},
\]

In an identical manner, part (ii) is implied from the following, which is true for all $r \geq 4$.
\[ 
\OPT_r(P_3) - \delta_2(r) = \frac{18}{r} - \frac{42}{r^2} - \frac{12}{r^3} + \frac{96}{r^4} - \frac{72}{r^5}  - \frac{24}{r^6} + \frac{24}{r^7} > \frac{18}{r} - \frac{42}{r^2} - \frac{12}{r^3},
\]
completing the proof.
\end{proof}

After assuming that $G$ is an extremal graph, and therefore close in edit-distance to $T_r(n)$, we will show that most vertices in $G$ must closely resemble a vertex appearing in the Tur\'{a}n graph. Given this fact, we will use Proposition~\ref{prop_lower_bound_path_vertex} to show that any vertex that does not look like this cannot be contained in enough copies of $P_3$ to justify $G$ being extremal. This will ultimately show that $G$ must be isomorphic to $T_r(n)$, since the removal/duplication process described in the proof of Proposition~\ref{prop_lower_bound_path_vertex} would otherwise increase the number of $P_3$ copies in $G$. 

\begin{proof}[Proof of Theorem~\ref{main_p4_thm}(ii)]
Let $r\geq 4$. Fix $\varepsilon > 0$ and assume that $n_0 = n_0(r)$ is large enough to satisfy the following conditions.
\begin{enumerate}
    \item[(i)] Any $K_{r+1}$-free graph $G$ on $n \geq n_0$ vertices with
\[ 
d(P_3,G) > OPT_r(P_3) - \varepsilon
\]
must satisfy $\text{Dist}(G,T_r(n)) \leq \frac{2}{r^{10}}n^2$.
\item[(ii)] $n_0 \geq 2r^4$ and is large enough to satisfy the conditions of Proposition~\ref{prop_lower_bound_path_vertex}.
\item[(iii)] $n_0$ is large enough to satisfy the conditions of Proposition~\ref{prop:n_0 bounds for the type 2 and 1 claims}.
\end{enumerate}

Let $G$ be an extremal graph on $n \geq n_0$ vertices. Recall that (i) means that we can transform $G$ into $T_r(n)$ by changing at most $\frac{2}{r^{10}}n^2$ adjacencies.
We will call each edge removed in the process of transforming $G$ into $T_r(n)$ a \emph{surplus edge}, and each added edge a \emph{missing edge}. Let $b(v)$ denote the total number of surplus edges and missing edges incident with a vertex $v$. If $v$ is a vertex for which $b(v) > \frac{1}{r^5} n$, then we say that $v$ is a \emph{bad vertex}. 

Partition the vertex set of $G$ into sets $X_1,X_2,\dots,X_r$ so that after changing all required adjancencies in $G$ the sets $X_1,X_2,\dots,X_r$ are the partite sets of $T_r(n)$. For the moment, move each bad vertex from its original set and place it into a new set $X_0$.
\begin{claim}\label{claim_ub_size_X0}
$|X_0| \leq \frac{1}{r^5} n.$
\end{claim}
\begin{proof}
Since $\text{Dist}(G,T_r(n)) \leq \frac{2}{r^{10}}n^2$ and each vertex $v \in X_0$ satisfies $b(v) > \frac{1}{r^5}n$, 
\[ |X_0|\cdot \frac{1}{r^5}n \leq \frac{1}{r^{10}} n^2.\]
Claim~\ref{claim_ub_size_X0} follows immediately.
\end{proof}

Now we will show that all surplus edges must be incident with at least one vertex in $X_0$. This will allow us to focus only on the bad vertices. For a finite collection of vertices $x_1,x_2,\dots,x_{\ell} \in V(G)$ let $N(x_1,x_2,\dots,x_{\ell})$ denote the \emph{common neighborhood} of $x_1,x_2,\dots,x_{\ell}$, which is the set of vertices in $V(G)$ adjacent to each of $x_1,x_2,\dots,x_{\ell}.$
\begin{claim}\label{claim_no_surplus_edges_inX_0}
There are no surplus edges in $V(G) \setminus X_0$. 
\end{claim}

\begin{proof}
Suppose for contradiction that for two vertices $u$ and $v$ in $X_j \setminus X_0$ are adjacent for some integer $j \in [r]$. By symmetry we may assume that $j = 1$. Since neither vertex is contained in $X_0$, both $u$ and $v$ are incident with at most $\frac{1}{r^5}n$ missing edges in $X_2 \setminus X_0$. This implies that there are at most $\frac{2}{r^5}n$ vertices in $X_2 \setminus X_0$ not contained in $N(u,v)$. Since Claim~\ref{claim_ub_size_X0} implies that we have moved at most $\frac{1}{r^5}n$ vertices from $X_2$ to $X_0$,
\[ 
| (N(u,v) \cap X_2) \setminus X_0 | \geq \left\lfloor\frac{n}{r}\right\rfloor - \left\lfloor\frac{3n}{r^5}\right\rfloor > 0.
\]
Let $w_2$ be one of the vertices contained in the set $(N(u,v) \cap X_2) \setminus X_0$. Then $uvw_2$ induces a triangle in $G$. Since $w_2$ is also only incident with $\frac{1}{r^5} n$ missing edges, we can apply an identical argument using $u,v,w_2$ and the set $X_3$ to show:
\[ 
| (N(u,v,w_2) \cap X_3) \setminus X_0 | \geq \left\lfloor\frac{n}{r}\right\rfloor - \left\lfloor\frac{4n}{r^5}\right\rfloor  > 0
,\]
implying that we can find some $w_3 \in X_3$ such that $uvw_2w_3$ induces a $K_4$ in $G$. Continuing this process for each $j \in \{4,\dots,r\}$, we can always select one vertex $w_j \in X_j$ in an identical manner so that $uvw_2\dots w_j$ induces a copy of $K_{j+1}$ in $G$. This is possible since 
\[ 
| (N(u,v,w_2,\dots,w_{j-1}) \cap X_j) \setminus X_0 | \geq \left\lfloor\frac{n}{r}\right\rfloor - \left\lfloor\frac{(j+1)n}{r^5}\right\rfloor > 0
\]
for each $j$. This, however, would imply that after selecting vertices $u,v,w_2,\dots,w_{r-1}$ that induce a copy of $K_r$,  
\[ 
| (N(u,v,w_2,\dots,w_{r-1}) \cap X_r) \setminus X_0 | \geq \left\lfloor\frac{n}{r}\right\rfloor - \left\lfloor\frac{(r+1)n}{r^5}\right\rfloor > 0. 
\]
Thus, we can select a vertex in $X_r$ that is adjacent to each of $u,v,w_2,\dots,w_{r-1}$. This, however, induces a copy of $K_{r+1}$ in $G$ which is a contradiction.
\end{proof}

For each $i \in [r]$, let $d_i(v) = | (N(v) \cap X_i)\setminus X_0 |$. 
We say that $v \in X_0$ is a \emph{type $2$ vertex} if $d_i(v) > 0$ for all $i = 1,\dots,r$. Otherwise, if there exists some $i \in [r]$ for which $d_i(v) = 0$, then $v$ is a \emph{type $1$ vertex}.

\begin{claim}\label{claim_upper_bound_degree_type_2}
If $v$ is a type $2$ vertex, then there exist $i,j \in [r]$ for which 
\[1 \leq d_i(v) \leq d_j(v) \leq \frac{1}{r^3}n.\]
\end{claim}
\begin{proof}
Suppose for contradiction that for all $i \in [r]$, $d_i(v) > \frac{1}{r^3}n$. By symmetry, we may assume that 
\[
\frac{1}{r^3}n < d_1(v) \leq d_2(v) \leq \cdots \leq d_r(v).
\]
Let $w_1 \in X_1$ be a neighbor of $v$. Then $d_i(w_1) \geq \frac{n}{r} - \frac{n}{r^5}$ for all integers $i \geq 2$ since $w \notin X_0$. This, along with Claim~\ref{claim_ub_size_X0}, implies that
\[ 
| (N(v,w_1) \cap X_2) \setminus X_0 | \geq \left\lfloor\frac{n}{r^3}\right\rfloor - \left\lfloor\frac{2n}{r^5}\right\rfloor > 0. 
\]
Using an argument identical to that in Claim~\ref{claim_no_surplus_edges_inX_0}, we can continue selecting vertices $w_j \in X_j$ for each $j \in \{3,\dots,r\}$ so that $vw_1w_2\dots w_j$ induces a copy of $K_{j+1}$. This is possible since for each $j \in \{3,\dots,r-1\}$, 
\[ 
| (N(v,w_1,w_2\dots,w_{j-1}) \cap X_j) \setminus X_0 | \geq \left\lfloor\frac{n}{r^3}\right\rfloor - \left\lfloor\frac{jn}{r^5}\right\rfloor > 0.
\]
This would imply, however, that $vw_1\dots w_r$ induces a copy of $K_{r+1}$. Since the above argument only relied on $d_1(v)$ being nonzero, and $v$ is a type $2$ vertex, this implies that $d_2(v) < \frac{1}{r^3} n$ completing the proof of Claim~\ref{claim_upper_bound_degree_type_2}. 
\end{proof}

Given $v \in G$ and a  path $P = vxyz$ or $P = xvyz$ of $P_3$ containing $v$, we say that $P$ is \emph{$v$-good} if none $x$, $y$, or $z$ is contained in $X_0$. The next claim will show that a type 2 vertex in $G$ would not be contained in enough copies of $P_3$ to justify $G$ being extremal. 

\begin{claim}\label{claim_no_type_2}
$G$ does not contain any type $2$ vertices. 
\end{claim}

\begin{proof}
Suppose that $v \in X_0$ is a type $2$ vertex. Then by symmetry, $d_1(v) \leq d_2(v) < \frac{1}{r^3}n$. Let $vu_1u_2u_3$ be a $v$-good path. We can count the number of these paths by considering the possible locations of $u_1$. The following list will provide the location of $u_1$, followed by the maximum number of paths of the form $vu_1u_2u_3$.

\begin{enumerate}
    \item If $u_1 \in X_1$ or $u_1 \in X_2$, then there are at most $\frac{2n}{r^3}$ ways to select $u_1$. Otherwise, there are at most $\frac{r-2}{r}n$ ways to select $u_1$. There are $\frac{(r-1)^2}{r^2}n^2$ ways to select $u_2$ and $u_3$ since the only requirement is that each vertex cannot be in the same set as its predecessor. This gives
    \[ \left(\frac{2}{r^3} + \frac{r-2}{r} \right) \cdot  \frac{(r-1)^2}{r^2}n^3 \]
    $v$-good copies of $P_3$ where $v$ is an end point.
\end{enumerate}

Next suppose that $u_1vu_2u_3$ is a $v$-good path. The maximum number of such paths can be counted by considering the locations of $u_1$ and $u_2$. In each case below, we give the location of $u_1$ and $u_2$, followed by the corresponding maximum number of $P_3$ subgraphs. 
\begin{enumerate}
    \item If $u_1,u_2 \in X_1$ or $u_1,u_2 \in X_2$, then there are at most $\frac{n^2}{r^6}$ ways to select each of $u_1$ and $u_2$ from either of the two sets. There are at most $\frac{r-1}{r}n$ ways to select $u_3$. If $u_1 \in X_1$ and $u_2 \in X_2$ or $u_1 \in X_2$ and $u_2 \in X_1$, then there are at most then there are at most $\frac{n^2}{r^6}$ ways to select $u_1$ and $u_2$ from each of their given sets. Again, there are at most $\frac{r-1}{r}n$ ways to select $u_3$. This accounts for at most
    \[ 
    \frac{4}{r^{6}}\cdot \frac{(r-1)}{r} \cdot n^3.
    \]
    copies of $P_3$.
    \item If exactly one of $u_1$ or $u_2$ is contained in $X_1 \cup X_2$, then there are at most $\frac{n}{r^3}$ ways to select that particular vertex. The vertex not in $X_1 \cup X_2$ can be selected from $(r-2)$ possible sets. Thus, there are $\frac{4(r-2)n^2}{r^4}$ ways to select $u_1$ and $u_2$. Finally, there are at most $\frac{r-1}{r}n$ ways to select $u_3$. This accounts for at most 
    \[ 
    \frac{4}{r^3}\cdot \frac{(r-1)(r-2)}{r^2} \cdot n^3 
    \]
    copies of $P_3$.
    \item If $u_1,u_2 \notin X_1 \cup X_2$, then there are at most $\frac{(r-2)(r-3)}{r^2}n^2$ ways to choose $u_1$ and $u_2$ if they are in different sets, and $\frac{(r-2)}{r^2}n^2$ ways to choose $u_1$ and $u_2$ if they are in the same set. As there are at most $\frac{r-1}{r}n$ ways to select $u_3$, this accounts for at most 
    \[ 
    \left( \frac{(r-1)(r-2)(r-3)}{r^3} + \frac{(r-1)(r-2)}{r^3} \right)n^3 
    \]
    copies of $P_3$.
\end{enumerate}
There are at most $\frac{2}{r^5}n^3$ subgraphs containing $v$ and at least one other vertex in $X_0$. Thus, combining each of the terms we have calculated, we get the following upper bound: 
\[ \nu(v,P_3) \leq \delta_2(r)\frac{n^3}{6}.\]
where $\delta_2(r)$ is taken from Proposition~\ref{prop:n_0 bounds for the type 2 and 1 claims}, which then implies the following: 
\begin{equation}\label{equation showing type 2 cannot exist}
    \OPT_r(P_3) \binom{n-1}{3} - \nu(v,P_3) \geq \left(\frac{18}{r} - \frac{42}{r^2} - \frac{12}{r^3} \right)n^3.
\end{equation}
It is straightforward to verify that for all $r  \geq  4$,
\[ 
\frac{18}{r} - \frac{48}{r^2} - \frac{12}{r^3} > \frac{1}{r^4}.
\]

Since \eqref{equation showing type 2 cannot exist} must be true of each type $2$ vertex and $G$ is assumed to be an extremal graph, Proposition~\ref{prop_lower_bound_path_vertex} implies that $G$ cannot contain any type $2$ vertices.
\end{proof}

By Claim~\ref{claim_no_type_2}, each $v \in X_0$ is a type $1$ vertex. We will now show that if $u$ and $v$ are two type $1$ vertices for which $d_i(v) = d_i(u) = 0$, then $u$ and $v$ cannot be adjacent. Specifically, we will prove that if $u$ and $v$ are adjacent, then one or the other is not contained in sufficiently many $P_3$ subgraphs to justify $G$ being extremal. Note this is slightly different from our approach to type $2$ vertices, as we will not disprove the existence of type $1$ vertices. 

\begin{claim}\label{claim_shared_neighborhood_type_one}
Suppose that $u$ and $v$ are two adjacent type 1 vertices for which $d_i(v) = d_i(u) = 0$. Then there exists some index $j \neq i$ for which 
\[ |N(u,v) \cap (X_j \setminus X_0)| \leq \frac{1}{r^3}n.\]
\end{claim}

\begin{proof}
By symmetry we may assume that $i = 1$. Suppose for contradiction that $|N(u,v) \cap (X_j \setminus X_0)| > \frac{1}{r^3}n$ for all $j \in \{2,\dots,r\}$. Using an argument identical to those in Claims~\ref{claim_no_surplus_edges_inX_0} and \ref{claim_upper_bound_degree_type_2}, select one vertex $w_j$ in $X_j$ for all $j = 2,\dots,r$, starting with $X_2$, so that $w_j \in N(u,v,w_2,\dots,w_{j-1}) \cap X_j$. This is possible since
\[
| (N(u,v,w_2,\dots,w_{j-1}) \cap X_j) \setminus X_0 | \geq \left\lfloor\frac{n}{r^3}\right\rfloor - \left\lfloor\frac{(j-1)n}{r^5}\right\rfloor > 0
\]
for all $j \in \{2,\dots,r-1\}$. After selecting vertices $w_2,\dots,w_r$ in this way, we once again obtain a copy of $K_{r+1}$ in $G$ which is a contradiction.
\end{proof}

\begin{claim}\label{claim_no_type_one}
If $u$ and $v$ are two type one vertices for which $d_i(v) = d_i(u) = 0$, then $u$ and $v$ are not adjacent. 
\end{claim}
\begin{proof}
By symmetry, Claim~\ref{claim_shared_neighborhood_type_one} implies that
\[ 
N(u,v) \cap (X_2 \setminus X_0)  \leq \frac{1}{r^3}n.
\]
Therefore without loss of generality, 
\[ 
|( N(v) \cap (X_2 \setminus X_0) ) \setminus N(u)| \leq \frac{r^2 - 1}{2r^3}n.
\]
Hence, 
\[ d_2(v) \leq \frac{r^2 +1}{2r^3} n.\]
Suppose that $vu_1u_2u_3$ is a $v$-good path. Similar to Claim~\ref{claim_no_type_2} we can count the number of such paths by considering the location of $u_1$. 
\begin{enumerate}
    \item If $u_1 \in X_2$ then there are $\frac{r^2 +1}{2r^3}n$ ways to choose $u_1$. Otherwise, there are $\frac{r-2}{r}n$ ways to choose $u_1$. Similar to before, there are $\frac{(r-1)^2}{r^2}n^2$ ways to choose $u_2$ and $u_3$. This accounts for at most  
    \[ 
    \left(\frac{r^2 +1}{2r^3} + \frac{r-2}{r} \right) \frac{(r-1)^2}{r^2}n^3
    \]
    copies of $P_3$ where $v$ is an end-vertex.
\end{enumerate}
Next we can count the number of $v$-good paths of the form $u_1vu_2u_3$ by considering the locations of $u_1$ and $u_2$. 
\begin{enumerate}
    \item If $u_1,u_2 \in X_2$, then there are at most $\left(\frac{r^2 +1}{2r^3}n\right)^2$ ways to select $u_1$ and $u_2$. There are $\frac{(r-1)}{r}n$ ways to select $u_3$. This gives an upper bound of
    \[ \left(\frac{r^2 +1}{2r^3}\right)^2\cdot \frac{(r-1)}{r} \cdot n^3 \]
    copies of $P_3$. 
    \item If exactly one of $u_1$ or $u_2$ is contained in $X_2$, then there are $\frac{r^2 +1}{2r^3}n$ to choose that specific vertex. Since neither of the remaining vertices can be contained in the same set as its neighbors, there are at most $\frac{(r-1)(r-2)}{r^2}n^2$ ways to choose the remaining vertices on the path. This gives at most
    \[ 2\cdot \frac{r^2 +1}{2r^3} \cdot \frac{(r-1)(r-2)}{r^2} \cdot n^3 \]
    copies of $P_3$. 
    \item If $u_1,u_2 \notin X_2$, then there are at most $\frac{(r-2)(r-3)}{r^2}n^2$ ways to choose $u_1$ and $u_2$ if they are in a different set and $\frac{(r-2)}{r^2}n^2$ ways if they are in the same set. There are $\frac{(r-1)}{r}n$ ways to select $u_3$, giving an upper bound of
    \[ \left( \frac{(r-1)(r-2)(r-3)}{r^3} + \frac{(r-1)(r-2)}{r^3} \right)n^3 \]
    copies of $P_3$. 
\end{enumerate}
Since there are at most $\frac{2}{r^5}n^3$ copies of $P_3$ containing $v$ and at least one other vertex in $X_0$, 
\[ \nu(v,P_3) \leq \delta_1(r)\frac{n^3}{6}.\]
Where $\delta_1(r)$ is taken from Proposition~\ref{prop:n_0 bounds for the type 2 and 1 claims}, which implies the following:
\[
    \OPT_r(P_3) \binom{n-1}{3} - \nu(v,P_3) \geq \left(\frac{9}{r} - \frac{39}{2r^2} \right)n^3.
\]
It is straightforward to verify that for all $r  \geq  4$,
\[ 
\frac{9}{r} - \frac{39}{2r^2} > \frac{1}{r^4}.
\]

Thus, by Proposition~\ref{prop_lower_bound_path_vertex}, vertex $v$ cannot exist in $G$ under the assumption that $G$ is extremal. Since $u$ and $v$ were arbitrarily chosen, this completes the proof of Claim~\ref{claim_no_type_one}. 
\end{proof}

\textbf{Proof of Theorem 1.3(ii), continued.} From
Claim~\ref{claim_no_type_one}, if two vertices $u$ and $v$ in $X_0$ have the property
that $d_i(u) = d_i(v) = 0$ for some $i \in [j]$, then $u$ and $v$ cannot be
adjacent. Thus, we can take each vertex in $X_0$ (since each vertex is a type
$1$ vertex) and place it in some partite set so that $G$ is a an $r$-partite
graph. Adding the necessary edges to make $G$ a complete $r$-partite graph,
however, would increase the number of $P_3$ subgraphs in $G$. As we have already
shown by Proposition~\ref{lem:rpartite_turan_best} that the Tur\'{a}n graph is
best possible among all complete $r$-partite graphs, this completes the proof of
Theorem~\ref{main_p4_thm}(ii).
\end{proof}

\section{Concluding Remarks}\label{sec:conclusion}

The main result in this paper follows a similar approach to that used
in~\cite{lidick2020maximizing}, which determined that the five cycle $C_5$ is
also $K_{r+1}$-Tur\'an-good for $r \ge 3$. It is likely that this method
could be applied to other graphs, perhaps including $P_4$ or $C_6$. However, as
the number of vertices in the target graph increases, the number of graphs
considered in the flag algebra step grow exponentially and the number of cases
in the stability result increase as well. 
Therefore, the authors believe
a different method will need to be used to
investigate the conjecture of Gerbner and Palmer that $P_{\ell}$ is
$K_{r+1}$-Tur\'{a}n-good for all values of $\ell$.

\section{Acknowledgements}

The authors would like to thank Bernard Lidick\'{y} for the use of his flag algebra program.

\bibliographystyle{abbrv}
\bibliography{P3}

\section{Appendix}

Link to SageMath code that can be used to verify equation~\eqref{eq:factoring to show OPT is what we want} in Theorem~\ref{main_p4_thm}(i):

\medskip

\href{https://drive.google.com/file/d/1CURgyBgoL54CyYg97t8u2aKkO3A0cNbW/view?usp=sharing}{\textcolor{blue}{https://drive.google.com/file/d/1CURgyBgoL54CyYg97t8u2aKkO3A0cNbW/view?usp=sharing}}
\end{document}